\documentclass[a4paper, 11pt]{amsart} 
\usepackage[
  margin=1.9cm,
  includefoot,
  footskip=30pt,
]{geometry}
\usepackage{amsmath,amssymb,amscd}
\usepackage{amsfonts}
\usepackage[english]{babel}
\usepackage[leqno]{amsmath}
\usepackage{amssymb,amsthm}
\usepackage{mathrsfs} 
\usepackage{amscd}
\usepackage{enumerate}
\usepackage{epsfig}
\usepackage{relsize}
\usepackage{layout}
\usepackage[usenames,dvipsnames]{xcolor}
\usepackage[backref=page]{hyperref}
\usepackage{tikz}
\hypersetup{
 colorlinks,
 citecolor=Green,
 linkcolor=Red,
 urlcolor=Blue}
\usepackage[matrix,arrow,tips,curve]{xy}
\input{xy}
\xyoption{all}
\definecolor{darkgreen}{rgb}{0.0, 0.7, 0.0}
\definecolor{purple}{rgb}{0.5, 0.0, 0.5}
\definecolor{red}{rgb}{0.8, 0.2, 0.0}

\newtheorem{thm}{Theorem}[section]
\newtheorem{bthm}{Theorem}

\newtheorem{lemma}[thm]{Lemma}

\numberwithin{equation}{section}
\theoremstyle{definition}
\newtheorem{defi}[thm]{Definition}

\theoremstyle{remark}

\newcommand{\Z}{\mathbb{Z}}

\newcommand{\Q}{\mathbb{Q}}

\newcommand{\Pic}{\operatorname{Pic}}

\DeclareMathOperator{\Ext}{{Ext}}

\def \P{\mathbb{P}}

\def \F{\mathcal F}

\def\I{{\mathcal J}}

\def \E{\mathcal E}
\def \G{\mathcal G}

\def\O{\mathcal O}

\def\M0{\mathcal M^0}

\DeclareMathOperator{\Sing}{{Sing}}

\begin{document}

\title[Non-existence of low rank Ulrich bundles on Veronese varieties]{Non-existence of low rank Ulrich bundles on Veronese varieties}

\author[A.F. Lopez, D. Raychaudhury]{Angelo Felice Lopez and Debaditya Raychaudhury}

\address{\hskip -.43cm Angelo Felice Lopez, Dipartimento di Matematica e Fisica, Universit\`a di Roma
Tre, Largo San Leonardo Murialdo 1, 00146, Roma, Italy. e-mail {\tt angelo.lopez@uniroma3.it}}

\address{\hskip -.43cm Debaditya Raychaudhury, Department of Mathematics, University of Arizona, 617 N Santa Rita Ave., Tucson, AZ 85721, USA. email: {\tt draychaudhury@math.arizona.edu}}

\thanks{The first author was partially supported by the GNSAGA group of INdAM and by the PRIN ``Advances in Moduli Theory and Birational Classification''.}

\thanks{{\it Mathematics Subject Classification} : Primary 14F06. Secondary 14J60, 14M10.}

\begin{abstract} 
We show that Veronese varieties of dimension $n \ge 4$ do not carry any Ulrich bundles of rank $r \le 3$. In order to prove this, we prove that a Veronese embedding of a complete intersection of dimension $m \ge 4$, which if $m=4$ is either $\P^4$ or has degree $d \ge 2$ and is very general and not of type $(2), (2,2)$, does not carry any Ulrich bundles of rank $r \le 3$.
\end{abstract}
\maketitle

\section{Introduction}

In the theory of vector bundles on a smooth irreducible variety $X \subset \P^N$, an open problem \cite{es} that has attracted attention lately is whether $X$ carries an Ulrich bundle $\E$, that is such that $H^i(\E(-p))=0$ for $i \ge 0, 1 \le p \le \dim X$. Once existence is proved, the next question is what is the Ulrich complexity of $(X,\O_X(1))$, namely the lowest possible rank of an Ulrich bundle. 

Perhaps one of the simplest but interesting cases  to be considered is when $X=\P^n$ and the embedding in $\P^N$ is given by $\O_{\P^n}(a)$, for some integer $a \ge 1$. Existence of Ulrich bundles is known for high rank: $n!$ by \cite[Thm.~3.1]{b} and $a^{\binom{n}{2}}$ by \cite[Cor.~5.7]{es}. As for lower rank, to set up the picture, let $\E$ be a rank $r$ Ulrich bundle for $(\P^n,\O_{\P^n}(a))$. If $a=1$ it is well-known \cite[Prop.~2.1]{es}, \cite[Thm.~2.3]{b} that $\E \cong \O_{\P^n}^{\oplus r}$. Also, if $n=1$ one easily sees that $\E \cong \O_{\P^1}(a-1)^{\oplus r}$. Things are different for $n \ge 2, a \ge 2$. First of all, one has that $r \ge 2$ and there are some strong numerical constraints, since (see for example \cite[Thm.~5.1]{es}) we have that
\begin{equation}
\label{cnec}
\chi(\E(\ell))=\frac{r}{n!}(\ell+a) \cdots (\ell+na) \in \Z \ \hbox{for every} \ \ell \in \Z.
\end{equation}
It follows by \cite[Cor.~5.3]{es} that if $p$ is any prime such that $p \vert a$ and $p^t \vert n!$, then $p^t \vert r$. 

The necessary condition \eqref{cnec} is easily translated into $2 \vert r(a-1)$ when $n=2$ and $6 \vert r(a^2-1)$ when $n=3$. If $n=2$, it follows by \cite[Thm.~1]{cmr1} (see also \cite[Thms.~6.1 and 6.2]{cg}) that it is in fact sufficient. If $n=3$, it was conjectured in \cite[Conj.~1.1]{cmr2} that it is again sufficient and this has been recently proved in \cite[Thm.~1]{fp}.

On the other hand, when $n \ge 4$, there seems to be an important difference. For example, consider the case $n=4$. We get by \cite[Cor.~5.3]{es} that $\gcd(a,6)=1$ for $r=2$ and $\gcd(a,2)=1$ for $r=3$. Similarly, if $n=5$, we find that $\gcd(a,30)=1$ for $r=2$ and $\gcd(a,10)=1$ for $r=3$. But assuming that the latter non-divisibility conditions on $a$ are satisfied, we have that \eqref{cnec} holds unconditionally.
 
Despite the fact that this seems to suggest that, when $n \ge 4$, Ulrich bundles of rank $2$ or $3$ might exist on Veronese varieties, we show that this is not the case. In fact we have:

\begin{bthm}
\label{Veronese}

Let $n \ge 4$ and $a \ge 2$. Then $(\P^n,\O_{\P^n}(a))$ does not carry Ulrich vector bundles of rank $r \le 3$.
\end{bthm}

The strategy to prove the above theorem is to consider the Veronese embedding $v_a(\P^n) \subset \P^N$, take hyperplane sections and use the fact that the restriction of an Ulrich bundle to the hyperplane section remains Ulrich. One then gets an Ulrich bundle on a Veronese embedding of a complete intersection of type $(a, \ldots, a)$ in $\P^n$. In order to handle these, we use Ulrich subvarieties (see Section \ref{due}) to show the following generalization of \cite[Thm.~2]{lr2}:

\begin{bthm} 
\label{ci}
Let $s \ge 1, a \ge 2, m \ge 4$ and let $X \subset \P^{m+s}$ be a smooth $m$-dimensional complete intersection of hypersurfaces of degrees $(d_1, \ldots, d_s)$ with $d_i \ge 1, 1 \le i \le s$ and degree $d$. Assume that one of the following holds:
\begin{itemize}
\item[(a)] $m \ge 5$, or
\item[(b)] $m =4$ and $d=1$, or 
\item[(c)] $m=4, d \ge 2$, $X$ is very general and
$(d_1, \ldots, d_s) \not\in \{(2, \underbrace{1, \ldots, 1}_\text{s-1}), (2, 2, \underbrace{1, \ldots, 1}_\text{s-2})\}$
(up to permu-

\vskip -.4cm \noindent tation). 
\end{itemize}
Then there are no rank $r \le 3$ Ulrich vector bundles with respect to $(X,\O_X(a))$. 
\end{bthm}

\section{Preliminaries}

\subsection{Notation and conventions} \hskip 3cm

Throughout the paper we work over the complex numbers. 

We will use the convention $\binom{\ell}{m}=\frac{\ell (\ell-1)\ldots (\ell-m+1)}{m!} \ \mbox{for} \ m \ge 1, \ell \in \Z$.
Note that $\binom{-\ell}{m}=(-1)^m\binom{\ell+m-1}{m}$ and $\chi(\O_{\P^m}(\ell))=\binom{\ell+m}{m}$.

\subsection{Generalities on (Ulrich) vector bundles} \hskip 3cm

We will need the following statement on vanishing of cohomology.

\begin{lemma} 
\label{van}
Let $a \ge 1$ be an integer, let $X \subset \P^N$ be a smooth irreducible variety of dimension $n \ge 1$ and let $\F, \G$ be two vector bundles on $X$. We have:
\begin{itemize}
\item[(i)] If $H^0(\G(2))=H^1(\G(1))=0$, then $H^1(\G)=0$.
\item[(ii)] If $H^0(\F(-a))=H^1(\F(-2a))=0$, then $H^1(\F(-j))=0$ for all $j \ge 2a$.
\end{itemize} 
\end{lemma}
\begin{proof} 
Let $Y \in |\O_X(1)|$. To see (i), observe that the exact sequence
$$0 \to \G(1) \to \G(2) \to \G(2)_{|Y} \to 0$$
implies that $H^0(\G(2)_{|Y})=0$. In particular we have that $\dim Y \ge 1$ and, since $H^0(\G(1)_{|Y}) \subseteq H^0(\G(2)_{|Y})=0$, we deduce that $H^0(\G(1)_{|Y})=0$. Then, the exact sequence
$$0 \to \G \to \G(1) \to \G(1)_{|Y} \to 0$$
implies that $H^1(\G)=0$. This proves (i). We now show (ii) by induction on $j$. If $j=2a$, then $H^1(\F(-j))=0$ by hypothesis. If $j \ge 2a+1$, set $\G=\F(-j)$. Then $H^1(\G(1))=H^1(\F(-j+1))=0$ by induction. Also, since $-j+2 \le 1-2a \le -a$ we have that 
$$H^0(\G(2))=H^0(\F(-j+2)) \subseteq H^0(\F(-a))=0.$$ 
Therefore (i) implies that $H^1(\F(-j))=0$ and this proves (ii). 
\end{proof}
We will often use the following well-known properties of Ulrich bundles.

\begin{lemma} 
\label{ulr}
Let $X \subset \P^N, L=\O_X(1)$ and let $\E$ be a rank $r$ Ulrich bundle. We have:
\begin{itemize}
\item[(i)] $\E_{|Y}$ is Ulrich on a smooth hyperplane section $Y$ of $X$.
\item[(ii)] If $n \ge 2$, then $c_2(\E) L^{n-2}=\frac{1}{2}[c_1(\E)^2-c_1(\E) K_X] L^{n-2}+\frac{r}{12}[K_X^2+c_2(X)-\frac{3n^2+5n+2}{2}L^2] L^{n-2}$.
\end{itemize} 
\end{lemma}
\begin{proof} 
See for example \cite[Lemma 3.2]{lr1}.
\end{proof}

\section{Ulrich subvarieties}
\label{due}

Ulrich subvarieties were defined in \cite{lr2}. We now recall some of the properties that they enjoy. 

First, we give a simplified version of \cite[Lemma 3.2]{lr2}, adapted to our purposes.

\begin{lemma} 
\label{zeta}
Let $n \ge 2$, let $X \subset \P^N$ be a smooth irreducible $n$-dimensional variety and let $\E$ be a rank $r \ge 2$ Ulrich bundle with $\det \E = \O_X(D)$. Then there is a subvariety $Z \subset X$ such that, if $Z \ne \emptyset$, we have:
\begin{itemize}
\item[(i)] $[Z]=c_2(\E)$.
\item[(ii)] If $r=2$, then $\omega_Z \cong \O_Z(K_X+D)$ and $c_2(Z)=c_2(X)_{|Z}-c_2(\E)_{|Z}+K_Z^2-K_Z {K_X}_{|Z}$.
\item[(iii)] If $r=3$, then 

$c_2(Z)=c_2(X)_{|Z}-c_2(\E)_{|Z}-c_1(\E)^2_{|Z}+K_Z {K_X}_{|Z}-{K_X^2}_{|Z}+2K_Z c_1(\E)_{|Z}-2{K_X}_{|Z} c_1(\E)_{|Z}$.
\end{itemize}
\end{lemma}
\begin{proof} 
See \cite[Lemma 3.2]{lr2}.
\end{proof}
Next, we recall the statement of \cite[Thm.~1]{lr2}. Let $Z \subset X$ be a Cohen-Macaulay, pure codimension $2$ subvariety and let $D$ be a divisor on $X$. The short exact sequence
$$0 \to  \I_{Z/X}(K_X+D) \to \O_X(K_X+D) \to \O_Z(K_X+D) \to 0$$
determines a coboundary map
$$\gamma_{Z,D} : H^{n-2}(\O_Z(K_X+D)) \to H^{n-1}(\I_{Z/X}(K_X+D))$$
whose dual, by Serre duality, is
$$\gamma_{Z,D}^*: \Ext^1_{\O_X}(\I_{Z/X}(D), \O_X) \to H^0(\omega_Z(-K_X-D)).$$
Then we have

\begin{thm} 
\label{1-1}
Let $X \subset \P^N$ be a smooth irreducible variety of dimension $n \ge 2$, degree $d \ge 2$ and let $D$ be a divisor on $X$. Then $(X,\O_X(1))$ carries a rank $r \ge 2$ Ulrich vector bundle $\E$ with $\det \E = \O_X(D)$ if and only if there is a subvariety $Z \subset X$ such that:
\begin{itemize}
\item[(a)] $Z$ is either empty or of pure codimension $2$,
\item[(b)] if $Z \ne \emptyset$ and either $r=2$ or $n \le 5$, then $Z$ is smooth (possibly disconnected),
\item[(c)] if $Z \ne \emptyset$ and $n \ge 6$, then $Z$ is either smooth or is normal, Cohen-Macaulay, reduced and with $\dim \Sing(Z) = n-6$,
\end{itemize}
and there is a $(r-1)$-dimensional subspace $W \subseteq \Ext^1_{\O_X}(\I_{Z/X}(D), \O_X)$ such that the following hold:
\begin{itemize}
\item[(i)] If $Z \ne \emptyset$, then $\gamma_{Z,D}^*(W)$ generates $\omega_Z(-K_X-D)$ (that is the multiplication map 

\noindent $\gamma_{Z,D}^*(W) \otimes \O_Z \to \omega_Z(-K_X-D)$ is surjective).
\item[(ii)] $H^0(K_X+nH-D)=0$.
\item[(iii)] $H^0(\I_{Z/X}(D-H))=0$.
\item[(iv)] If $n \ge 3$, then $H^i(\I_{Z/X}(D-pH))=0$ for $1 \le i \le n-2, 1 \le p \le n$.
\item[(v)] $(-1)^{n-1}\chi(\I_{Z/X}(D-pH))=(r-1) \chi(K_X+pH)$, for $1 \le p \le n$.
\item[(vi)] $\delta_{Z, W, -nH} : H^{n-1}(\I_{Z/X}(D-nH)) \to W^* \otimes H^n(-nH)$ is either injective or surjective.
\end{itemize}
Moreover the following exact sequences hold
\begin{equation}
\label{est}
0 \to W^* \otimes \O_X \to \E \to \I_{Z/X}(D) \to 0
\end{equation}
and, if $Z \ne \emptyset$,
$$0 \to \O_X(-D) \to \E^* \to  W \otimes \O_X \to \omega_Z(-K_X-D) \to 0.$$
\end{thm}
Then we have

\begin{defi} 
\label{ulrsub}
Let $n \ge 2, d \ge 2, r \ge 2$ and let $D$ be a divisor on $X$. An {\it Ulrich subvariety of $X$} is a subvariety $Z \subset X$ carrying a $(r-1)$-dimensional subspace $W \subseteq \Ext^1_{\O_X}(\I_{Z/X}(D), \O_X)$ such that properties (a)-(c) and (i)-(vi) of Theorem \ref{1-1} hold.
\end{defi}
Hence, if $n \ge 2, d \ge 2, r \ge 2$, to each rank $r$ Ulrich vector bundle $\E$ with $\det \E = \O_X(D)$ one can associate as in Theorem \ref{1-1} an Ulrich subvariety $Z$. In particular, by its construction, $Z$ satisfies the properties of Lemma \ref{zeta}.
 
\section{Veronese embeddings of complete intersections}

In this section we study Ulrich bundles on Veronese embeddings of complete intersections.

Given integers $d_i \ge 1, 1 \le i \le s$, we set 
$$d = \prod\limits_{i=1}^s d_i, S = \sum\limits_{i=1}^s d_i \ \hbox{and} \ S' =  \begin{cases} 0 & \hbox{if } s=1 \\ \sum\limits_{1 \le i < j \le s}d_id_j & \hbox{if } s \ge 2 \end{cases}.$$
Then we have
\begin{lemma} 
\label{genci}
Let $s \ge 1, r \ge 2, a \ge 2, m \ge 3$ and let $X \subset \P^{m+s}$ be a smooth $m$-dimensional complete intersection of hypersurfaces of degrees $(d_1, \ldots, d_s)$ with $d_i \ge 1, 1 \le i \le s$ and degree $d$. Let $H \in |\O_X(1)|$. Let $\E$ be a rank $r$ Ulrich vector bundle for $(X,\O_X(a))$ and let $Z \subset X$ be the associated Ulrich subvariety, as in Theorem \ref{1-1} applied to the Veronese embedding $v_a(X) \subset \P^N$. Then $Z$ is irreducible, of dimension $m-2$, smooth when $r=2$ or when $m \le 5$ and:
\begin{itemize}
\item[(i)] $K_X=(S-s-m-1)H$.
\item[(ii)] $c_2(X)=\left[\binom{m+s+1}{2}+S(S-s-m-1)-S'\right]H^2$.
\item[(iii)] $c_1(\E)=uH$ where $u=\frac{r}{2}[(m+1)(a-1)+S-s]$.
\item[(iv)] $\begin{aligned}[t]
\deg_H(Z)=\frac{rd}{24}[&-4+6a-2a^2-7m+12am-5a^2m-3m^2+6am^2-3a^2m^2+3r-6ar+3a^2r+\\
& +6mr-12amr+6a^2mr+3m^2r-6am^2r+3a^2m^2r-7s+6as-6ms+6ams+\\
& +6rs-6ars+6mrs-6amrs-3s^2+3rs^2+6S-6aS+6mS-6amS-6rS+\\
& +6arS-6mrS+6amrS+6sS-6rsS-2S^2+3rS^2-2S')].
\end{aligned}$
\item[(v)] $\begin{aligned}[t]
\chi(\O_Z(\ell))=& \binom{\ell+m+s}{m+s}+(-1)^{m+1}\frac{rd}{m!}(u-\ell-a) \cdots (u-\ell-ma)+(-1)^{m+s}(r-1)\binom{u-\ell-1}{m+s}+ \\ & \hskip -2cm + \sum_{k=1}^s(-1)^{k+m+s}\sum_{1\le i_1<\ldots<i_k\le s}\left[\binom{d_{i_1}+\ldots+d_{i_k}-\ell-1}{m+s}+(r-1)\binom{d_{i_1}+\ldots+d_{i_k}+u-\ell-1}{m+s}\right].
\end{aligned}$
\end{itemize}
Moreover suppose that one of the following holds:
\begin{itemize}
\item[(1)] $m \ge 5$, or 
\item[(2)] $m=4, d=1$, or
\item[(3)] $m=4$, $X=X' \cap F$, where $X' \subset \P^{5+s}$ is a smooth complete intersection, $F \subset \P^{5+s}$ is a hypersurface of degree $a$ and $\E = \E'_{|X}$, where $\E'$ is a vector bundle on $X'$, or
\item[(4)] $m=4, d \ge 2$, $X$ is very general and $(d_1, \ldots, d_s) \not\in \{(2, \underbrace{1, \ldots, 1}_\text{s-1}), s \ge 1; (2, 2, \underbrace{1, \ldots, 1}_\text{s-2}), s \ge 2\}$  

\vskip -.4cm \noindent (up to permutation). 
\end{itemize}
Then
\begin{itemize}
\item[(vi)] $c_2(\E)=eH^2$ with 

$\begin{aligned}[t] 
e=\frac{r}{24}[&-4+6a-2a^2-7m+12am-5a^2m-3m^2+6am^2-3a^2m^2+3r-6ar+3a^2r+\\
& +6mr-12amr+6a^2mr+3m^2r-6am^2r+3a^2m^2r-7s+6as-6ms+6ams+\\
& +6rs-6ars+6mrs-6amrs-3s^2+3rs^2+6S-6aS+6mS-6amS-6rS+\\
& +6arS-6mrS+6amrS+6sS-6rsS-2S^2+3rS^2-2S')].
\end{aligned}$
\end{itemize}
\end{lemma}
\begin{proof}
(i) and (ii) follow from the tangent and Euler sequence of $X \subset \P^{m+s}$. By Lefschetz's theorem (see for example \cite[Thm.~2.1]{h}), we have that $\Pic(X) \cong \Z H$. Then (iii) follows by \cite[Lemma 3.2]{lo}. Note that $\deg v_a(X)=(aH)^m=a^md \ge 2$, thus \cite[Rmk.~4.3]{lr2} applies. Since $H^1(\O_X(-u))=0$ we have that $Z \ne \emptyset$ by \cite[Rmk.~4.3(i)]{lr2}, hence $Z$ is of dimension $m-2$, smooth when $r=2$ or when $m \le 5$. Also, $u \ge 2a$ and $H^i(\E(-pa))=0$ for $i \ge 0, 1 \le p \le m$, hence $H^1(\E(-u))=0$ by Lemma \ref{van}(ii). On the other hand, $H^2(\O_X(-u))=0$, hence $Z$ is irreducible by \cite[Rmk.~4.3(vi)]{lr2}. Next, Lemma \ref{zeta}(i) gives that $[Z]=c_2(\E)$, hence (iv) follows by (i)-(iii) and Lemma \ref{ulr}(ii) with $L=aH$. As for (v), observe first that $\chi(\E(\ell))$ is a polynomial in $\ell$ of degree $m$ with leading coefficient $\frac{rd}{m!}$. On the other hand, as $\E$ is Ulrich for $(X,\O_X(a))$, we have that $\chi(\E(-pa))=0$ for $1 \le p \le m$. Therefore
$$\chi(\E(\ell))=\frac{rd}{m!}(\ell+a) \cdots (\ell+ma).$$
Now, the exact sequence \eqref{est} twisted by $\O_X(\ell-u)$ gives
$$\chi(\O_Z(\ell))=\chi(\O_X(\ell))-\chi(\I_{Z/X}(\ell)=\chi(\O_X(\ell))-\chi(\E(\ell-u))+(r-1)\chi(\O_X(\ell-u))$$
and computing $\chi(\O_X(\ell))$ and $\chi(\O_X(\ell-u))$ with the Koszul resolution of $\I_{X/\P^{m+s}}$, we get (v).  
Finally, to see (vi), we claim that under any of the hypotheses (1), (2), (3) or (4), the following holds:
\begin{equation}
\label{meglio}
\exists e \in \Z \ \hbox{such that} \ c_2(\E) = eH^2. 
\end{equation}
In fact, if $m \ge 5$, we have by Lefschetz's theorem (see for example \cite[Thm.~2.1]{h}) that $H^4(X,\Z) \cong \Z H^2$. If $m=4$ and $d=1$, we have that $X \cong \P^4$. Hence \eqref{meglio} holds under either one of the hypotheses (1) or (2). Under hypothesis (3), we have as above that $H^4(X',\Z) \cong \Z (H')^2, H' \in |\O_{X'}(1)|$, hence $c_2(\E') = e(H')^2$ on $X'$, for some $e \in \Z$. Therefore $c_2(\E) = c_2(\E'_{|X})=eH^2$, so that \eqref{meglio} holds under hypothesis (3). Also, under hypothesis (4), we know again by Noether-Lefschetz's theorem (see for example \cite[Thm.~1.1]{sp}) that every algebraic cohomology class of codimension $2$ in $X$ is in $\Z H^2$. Since $[Z]=c_2(\E)$ by Lemma \ref{zeta}(i), we have that \eqref{meglio} holds under hypothesis (4).
Now, since $[Z]=c_2(\E)$ using (iv) and \eqref{meglio}, we have that $e= \frac{\deg_H(Z)}{d}$ is as stated in (vi).
\end{proof}

With the above lemma at hand, we now show Theorem \ref{ci}.

\renewcommand{\proofname}{Proof of Theorem \ref{ci}}
\begin{proof}
First, we dispose of the case $r=1$. Since $\Pic(X) \cong \Z H$ by Lefschetz's theorem, let $b \in \Z$ and let $\O_X(b)$ be an Ulrich bundle with respect to $(X,\O_X(a))$. Then $H^0(\O_X(b-a))=H^m(\O_X(b-ma))=0$, that is, by Lemma \ref{genci}(i), $H^0(\O_X(b-a))=H^0(\O_X(S-s-m-1-b+ma))=0$. But this gives the contradiction $m(a-1)+S-s \le b \le a-1$. 

Hence, from now on, we can assume that $r \in \{2, 3\}$.

Note that by adding some $1$'s to $(d_1, \ldots, d_s)$, we can always assume that $s \ge 4$.

Suppose that we have an Ulrich bundle of rank $r$ for $(X,\O_X(a))$ and consider the Veronese embedding $v_a(X) \subset \P^N$. If $m \ge 5$, taking hyperplane sections and using Lemma \ref{ulr}(i), it will be enough to show that, on the $4$-dimensional section of $v_a(X)_4$ of $v_a(X)$, there are no Ulrich bundles of rank $r$. 

Now $v_a(X)_4$ is isomorphic to a smooth complete intersection $\widetilde X \subset \P^{4+ \tilde s}$ of type $(d_1, \ldots, d_{\tilde s})=(d_1, \ldots, d_s, a, \ldots, a)$, with $\tilde s=s+m-4$ and we have an Ulrich bundle of rank $r$ for $(\widetilde X,\O_{\widetilde X}(a))$.

With an abuse of notation, let us call again $X$ the above $4$-dimensional section $\widetilde X$ and $(d_1, \ldots, d_s)$ its degrees. Hence we have that $X \subset \P^{4+s}$ is a smooth complete intersection of hypersurfaces of degrees $(d_1, \ldots, d_s)$ with $s \ge 4, d_i \ge 1, 1 \le i \le s$.

Let $\E$ be an Ulrich bundle of rank $r$ for $(X,\O_X(a))$ and let $Z \subset X$ be the associated smooth irreducible surface by Lemma \ref{genci}. 

Observe that, under hypothesis (a) (respectively (b), resp. (c)) of the theorem, we have that condition (3) (respectively (2), resp. (4)) of Lemma \ref{genci} hold. In any case, we deduce that Lemma \ref{genci}(vi) holds.

Let $H \in |\O_X(1)|$ and set $H_Z = H_{|Z}$. 

Assume that $r=2$.

By Lemma \ref{genci}(iii), (iv), (v) and (vi) we see that 
\begin{equation}
\label{u}
u=5(a-1)+S-s
\end{equation}
\begin{equation}
\label{e}
e=\frac{1}{12}(70-150a+80a^2+29s-30as+3s^2-30S+30aS-6sS+4S^2-2S')
\end{equation}
\begin{equation}
\label{deg}
\deg_H(Z)=H_Z^2=\frac{d}{12}(70-150a+80a^2+29s-30as+3s^2-30S+30aS-6sS+4S^2-2S')
\end{equation}
and
$$\begin{aligned}[t]
\chi(\O_Z)=& 1-\frac{d}{12}(u-a) \cdots (u-4a)+ (-1)^s\binom{u-1}{s+4}+ \\ & \hskip -3cm + \sum_{k=1}^s(-1)^{k+s}\sum_{1\le i_1<\ldots<i_k\le s}\left[\binom{d_{i_1}+\ldots+d_{i_k}-1}{s+4}+\binom{d_{i_1}+\ldots+d_{i_k}+u-1}{s+4}\right]
\end{aligned}.$$
In the notation \eqref{effe} of the functions in the appendix, this is just
\begin{equation}
\label{prima}
\chi(\O_Z)=f_{a,4,s,2,0}(d_1,\ldots,d_s).
\end{equation}
Next, we have by Lemma \ref{zeta}(ii), \eqref{u} and Lemma \ref{genci}(i) that
\begin{equation}
\label{k}
K_Z=[2S-2s+5(a-2)]H_Z
\end{equation}
so that
\begin{equation}
\label{k2}
K_Z^2=(100-100a+25a^2+40s-20as+4s^2-40S+20aS-8sS+4S^2)\deg_H(Z).
\end{equation}
Using Lemma \ref{zeta}(ii), Lemma \ref{genci}(i)-(ii), \eqref{u}, \eqref{e},  \eqref{k} and \eqref{k2} we get
\begin{equation}
\label{c2}
c_2(Z)=\frac{1}{12}(650-750a+220a^2+265s-150as+27s^2-270S+150aS-54sS+32S^2-10S')\deg_H(Z).
\end{equation}
Then Noether's formula $Z$, $\chi(\O_Z)=\frac{1}{12}[K_Z^2+c_2(Z)]$ gives, using \eqref{deg}, \eqref{k2} and \eqref{c2}, that
$$\begin{aligned}[t]
\chi(\O_Z)=\frac{5d}{1728}[&25900-82800a+95380a^2-46800 a^3+8320a^4+21160s-50220as+38336a^2s\\
& -9360a^3s+6481s^2-10152as^2+3852a^2s^2+882s^3-684as^3+45s^4-21600S+\\
&+50760aS-38520a^2S+9360a^3S-13140sS+20412asS-7704a^2sS-2664s^2S+\\
&+2052as^2S-180s^3S+7100S^2-10800aS^2+4036a^2S^2+2860sS^2-2160asS^2+\\
&+288s^2S^2-1080S^3+792aS^3-216sS^3+64S^4-880S'+1080aS'-368a^2S'\\
& -356sS'+216asS'-36s^2S'+360SS'-216aSS'+72sSS'-40S^2S'+4(S')^2].
\end{aligned}$$
In the notation \eqref{gi} of the appendix, this is just
\begin{equation}
\label{seconda}
\chi(\O_Z)=g_{a,4,s}(d_1,\ldots,d_s).
\end{equation}
Therefore \eqref{prima} and \eqref{seconda} imply that 
$$g_{a,4,s}(d_1,\ldots,d_s)-f_{a,4,s,2,0}(d_1,\ldots,d_s)=0$$
that is, by Lemma \ref{gl4}(1) of the appendix,
$$m_{1^s}(s)(d_1,\ldots,d_s)v_{s,a,8}(d_1,\ldots,d_s)=0$$
or, equivalently,
$$dv_{s,a,8}(d_1,\ldots,d_s)=0$$
contradicting Lemma \ref{cg}.

Next, assume that $r=3$.

By Riemann-Roch we see that
\begin{equation}
\label{hk}
K_ZH_Z=-2\chi(\O_Z(1))+2\chi(\O_Z)+\deg_H(Z).
\end{equation}
Now Lemma \ref{genci} gives, in the notation \eqref{effe} of the functions in the appendix, that 
\begin{equation}
\label{chi}
\chi(\O_Z(\ell))=f_{a,4,s,3,\ell}(d_1,\ldots,d_s)
\end{equation}
and that
$$\deg_H(Z)=\frac{d}{8}(145-300a+155a^2+59s-60as+6s^2+(-60+60a-12s)S+7S^2-2S')$$
that is, in the notation \eqref{gi}, that
\begin{equation}
\label{deg2}
\deg_H(Z)=\delta_s(d_1,\ldots,d_s)
\end{equation}
and therefore, in the notation \eqref{gi}, \eqref{hk} becomes
\begin{equation}
\label{ack}
K_ZH_Z=-2f_{a,4,s,3,1}(d_1,\ldots,d_s)+2f_{a,4,s,3,0}(d_1,\ldots,d_s)+\delta_s(d_1,\ldots,d_s)=h_s(d_1,\ldots,d_s).
\end{equation}
On the other hand, we have by \cite[Rmk.~4.3(ix)]{lr2} and Lemma \ref{genci} that $[K_Z-\frac{5}{2}(S-s+3a-5)H_Z]^2=0$, so that, using \eqref{deg2}, \eqref{ack} and the notation \eqref{gi}, we get
\begin{equation}
\label{k2'}
\begin{aligned}
K_Z^2 &= 5(S-s+3a-5)K_ZH_Z-\frac{25}{4}(S-s+3a-5)^2\deg_H(Z)=\\
& =5(S-s+3a-5)h_s(d_1,\ldots,d_s)-\frac{25}{4}(S-s+3a-5)^2\delta_s(d_1,\ldots,d_s)= k_s(d_1,\ldots,d_s).
\end{aligned}
\end{equation}
Next, we get by Lemma \ref{zeta}(iii) and Lemma \ref{genci}(i)-(iii) and (vi), using also the notation \eqref{gi}, that
\begin{equation}
\label{c2z'}
\begin{aligned}
c_2(Z) & = \frac{1}{8}[-1315+1800a-605a^2-523s+360as-52s^2+520S-360aS+104sS-49S^2-6 S']H_Z^2\\
& +(4S-4s-20+15a)K_ZH_Z=\\
& = \frac{1}{8}[-1315+1800a-605a^2-523s+360as-52s^2+520S-360aS+104sS\\
& \hskip .8cm -49S^2-6 S']\delta_s(d_1,\ldots,d_s)+(4S-4s-20+15a)h_s(d_1,\ldots,d_s)=\\
& = c_s(d_1,\ldots,d_s).
\end{aligned}
\end{equation}
Hence \eqref{k2'}, \eqref{c2z'} and Noether's formula, using also the notation \eqref{gi},  give
\begin{equation}
\label{secondachi}
\chi(\O_Z)=\frac{1}{12}(K_Z^2+c_2(Z))=\frac{k_s(d_1,\ldots,d_s)+c_s(d_1,\ldots,d_s)}{12}=\chi'_s(d_1,\ldots,d_s).
\end{equation}
Thus we get, by \eqref{chi} and \eqref{secondachi} we have that 
$$\chi'_s(d_1,\ldots,d_s)-f_{a,4,s,3,0}(d_1,\ldots,d_s)=0$$
that is, using Lemma \ref{gl4}(2),
$$m_{1^s}(s)(d_1,\ldots,d_s)v_{s,a,9}(d_1,\ldots,d_s)=0$$
or, equivalently,
$$dv_{s,a,9}(d_1,\ldots,d_s)=0$$
contradicting Lemma \ref{cg}.

This concludes the proof in the case $r=3$ and therefore also ends the proof of the theorem.
\end{proof}
\renewcommand{\proofname}{Proof}

\section{Proof of Theorem \ref{Veronese}}

In this section we prove our main theorem. 

\renewcommand{\proofname}{Proof of Theorem \ref{Veronese}}
\begin{proof}
Let $\overline \E$ be an Ulrich bundle of rank $r \le 3$ for $(\P^n,\O_{\P^n}(a))$. By Theorem \ref{ci}(b) we see that the case $n=4$ cannot occur. Hence we assume from now on that $n \ge 5$.

We can consider $\overline \E$ as an Ulrich bundle for $(v_a(\P^n), \O_{v_a(\P^n)}(1))$, where $v_a(\P^n) \subset \P^N$ is the $a$-Veronese embedding. Choosing $n-4$ general hyperplanes $H_i$ in $\P^N$, we get, by Lemma \ref{ulr}(i), a rank $r$ Ulrich bundle $\E'=\overline \E_{|X'}$ on $X'=v_a(\P^n)\cap H_1 \cap \ldots \cap H_{n-4}$ with respect to $\O_{X'}(1)=\O_{\P^N}(1)_{|X'}$. On the other hand, $X'$ is isomorphic to a general $4$-dimensional smooth complete intersection $X \subset \P^n$ of type $(a, \ldots, a)$ and, via this isomorphism, $\E'$ corresponds to a rank $r$ Ulrich bundle $\E$ on $X$ with respect to $\O_X(a)$. We have then obtained a nonempty open subset $U$ in the parameter space $M$ of complete intersections $X \subset \P^n$ of type $(a, \ldots, a)$. Since $U$ cannot be contained in a countable union of proper closed subvarieties of $M$, we deduce that we can find a very general $4$-dimensional smooth complete intersection $X \subset \P^n$ of type $(a, \ldots, a)$ carrying a rank $r$ Ulrich bundle $\E$ with respect to $\O_X(a)$. 

Therefore, setting $s=n-4$, we have that $X \subset \P^{4+s}$ is a very general $4$-dimensional smooth complete intersection of type $(d_1, \ldots, d_s)=(a, \ldots, a)$ carrying a rank $r$ Ulrich bundle $\E$ with respect to $\O_X(a)$.

Then we get a contradiction by Theorem \ref{ci} unless $a=2$ and $s=1, 2$, that is $n=5, 6$. But in the latter two cases we have a contradiction by \cite[Cor.~5.3]{es}.

This concludes the proof.
\end{proof}
\renewcommand{\proofname}{Proof}

\null\vfill

\eject

\appendix
\section{Symmetric functions associated to Veronese embeddings of complete intersections}
\label{app}

Given a smooth complete intersection $X \subset \P^{m+s}$ of hypersurfaces of degrees $(d_1, \ldots, d_s)$, a rank $r \ge 2$ Ulrich vector bundle $\E$ on $X$ with respect to $\O_X(a)$ and an Ulrich subvariety $Z$ associated to $\E$, we have some natural symmetric functions of $(d_1, \ldots, d_s)$ as in Lemma \ref{genci} and as in the proof of Theorem \ref{ci}. In this section we will lay out the necessary calculations related to them. Several calculations have been performed by Mathematica. The corresponding codes can be found in \cite{lr3}.

\begin{defi}
\label{effe}
Given integers $m \ge 1, s \ge 1, r \ge 2, \ell$, consider the polynomials in $\Q[x_1,\ldots, x_s]$ given by 

$$a_s(\ell, x_1, \ldots, x_s) = \binom{\ell+m+s}{m+s}+\sum_{k=1}^s(-1)^{k+m+s}\sum_{1\le i_1<\ldots<i_k\le s}\binom{x_{i_1}+\ldots+x_{i_k}-\ell-1}{m+s}$$
and
$$b_s(x_1, \ldots, x_s)=(-1)^{m+1}\frac{r}{m!}(\prod_{i=1}^sx_i)\prod\limits_{j=1}^m\left[\frac{r}{2}[(m+1)(a-1)+\sum\limits_{i=1}^s x_i-s]-\ell-ja\right]$$
Next, we set 
$$f_{a,m,s,r,\ell}=a_s(\ell, x_1, \ldots, x_s)+(r-1)a_s(\ell-\frac{r}{2}[(m+1)(a-1)+\sum\limits_{i=1}^s x_i-s], x_1, \ldots, x_s)+b_s(x_1, \ldots, x_s).$$
\end{defi}
Explicitly we have
$$\begin{aligned}[t]
& f_{a,m,s,r,\ell}= \binom{\ell+m+s}{m+s}+(-1)^{m+1}\frac{r}{m!}(\prod_{i=1}^sx_i)\prod\limits_{j=1}^m\left[\frac{r}{2}[(m+1)(a-1)+\sum\limits_{i=1}^s x_i-s]-\ell-ja\right]+\\
& + (-1)^{m+s}(r-1)\binom{\frac{r}{2}[(m+1)(a-1)+\sum\limits_{i=1}^s x_i-s]-\ell-1}{m+s}+\\
& +\sum_{k=1}^s(-1)^{k+m+s}\sum_{1\le i_1<\ldots<i_k\le s}\binom{x_{i_1}+\ldots+x_{i_k}-\ell-1}{m+s}+\\
& +(r-1)\sum_{k=1}^s(-1)^{k+m+s}\sum_{1\le i_1<\ldots<i_k\le s}\binom{x_{i_1}+\ldots+x_{i_k}+\frac{r}{2}[(m+1)(a-1)+\sum\limits_{i=1}^s x_i-s]-\ell-1}{m+s}.
\end{aligned}$$

\begin{lemma}
\label{tf0}
\null \hskip 1cm
\begin{itemize}
\item[(1)] $f_{a,m,s,r,\ell}$ is symmetric in $x_1,\ldots, x_s$.
\item[(2)] For any $1 \le k \le s$, the following identity holds in $\Q[x_1,\ldots,x_k]$: 
$$f_{a,m,k,r,\ell}(x_1, \ldots, x_k) = f_{a,m,s,r,\ell}(x_1,\ldots, x_k,1, \ldots, 1).$$
\item[(3)] $x_i \mid f_{a,m,s,r,\ell}$ for all $1\le i\le s$.
\end{itemize}
\end{lemma}
\begin{proof}
Same as \cite[Lemmas A.2, A.3]{lr2}.
\end{proof}

We will now express the symmetric polynomials $f_{a,m,s,r,\ell}$ in terms of monomial symmetric polynomials. For this we will use some properties of them, for which we refer for example to \cite[\S 1]{eg}.

\begin{defi}
Let $s \ge 1$ be an integer and let $x_1,\ldots, x_s$ be indeterminates. Given a partition $\lambda = \{\lambda_1, \ldots, \lambda_k\}$ with $\lambda_1 \ge \lambda_2 \ge \ldots \ge \lambda_k \ge 1$, if $k \le s$ we let $m_{\lambda}(s)$ be the monomial symmetric polynomial in $x_1,\ldots,x_s$ corresponding to $\lambda$, while if $k > s$ we set $m_{\lambda}(s)=0$.
\end{defi}
We will also write $m_{\lambda}(s)=m_{\lambda_1 \ldots \lambda_k}$. We denote by $\{1^k\}$ the partition $\{1,\ldots,1\}$ of $k$ and we set $m_{1^0}(s)=1$. For example
$$m_h(s)=\sum\limits_{i=1}^s x_i^h \ \hbox{for} \ h \ge 1 \ \hbox{and} \ m_{1^s}(s)=\prod_{i=1}^sx_i.$$

We will consider below the following $\Q$-basis of the vector space of symmetric polynomials with rational coefficients and of degree at most $4$ in $s$ variables:
\begin{equation} 
\label{base}
e=\{m_4(s), m_{31}(s), m_{22}(s), m_{211}(s), m_{1111}(s), m_3(s), m_{21}(s), m_{111}(s), m_2(s), m_{11}(s), m_1(s), 1\}.
\end{equation}
We can now express any of the functions $f \in \{f_{a,4,s,2,0}, f_{a,4,s,3,0}, f_{a,4,s,3,1}\}$ in terms of the above basis as
\begin{equation}
\label{inbase}
\begin{aligned} 
f =  \frac{m_{1^s}(s)}{M}[& a_1m_4(s)+a_2m_{31}(s)+a_3m_{22}(s)+a_4m_{211}(s)+a_5m_{1111}(s)+a_6m_3(s)+a_7m_{21}(s)+\\
& +a_8m_{111}(s)+a_9m_2(s)+a_{10}m_{11}(s)+a_{11}m_1(s)+a_{12}]
\end{aligned}
\end{equation}
with the coefficients given by the following
\begin{lemma}
\label{gl1} 
Let $f \in \{f_{a,4,s,2,0}, f_{a,4,s,3,0}, f_{a,4,s,3,1}\}$. For all $s \ge 4$ the coefficients in \eqref{inbase} are:
\begin{itemize}
\item[(i)] For $f_{a,4,s,2,0}$ we have $M=360$ and
\begin{enumerate}
\item $a_1 = 66, a_2 = 225, a_3 = 320, a_4 = 600, a_5 = 1125$
\item $a_6 = 75(-15+11a-3s), a_7 = 150(-20+15a-4s)$
\item $a_8 = 225(-25+19a-5s)$
\item $a_9 = 10(740-1125a+420a^2+298s-225as+30s^2)$
\item $a_{10} = \frac{75}{2}(370-570a+214a^2+149s-114as+15s^2)$
\item $a_{11} = \frac{75}{2}(-600+1410a-1070a^2+260a^3-365s+567as-214a^2s-74s^2+57as^2-5s^3)$
\item $\begin{aligned}[t] a_{12} = \frac{1}{8}(&215760-690000a+795000a^2-390000a^3+69240a^4+176302s-418500as\\
& +319500a^2s-78000a^3s+54005s^2-84600as^2+32100a^2s^2+7350s^3-5700as^3+\\
& +375s^4).
\end{aligned}$
\end{enumerate}
\item[(ii)] For $f_{a,4,s,3,0}$ we have $M=1920$ and
\begin{enumerate}
\item $a_1=1683, a_2=6060, a_3=8770, a_4=16860, a_5=32400$
\item $a_6=-30300+24600a-6060s, a_7=-84300+69000a-16860s$
\item $a_8=-162000+133200a-32400 s$
\item $a_9=209050-345000a+140850a^2+83960s-69000as+8430s^2$
\item $a_{10}=401700-666000a+272700a^2+161340s-133200as+16200s^2$
\item $\begin{aligned}[t]
a_{11}=& -658500+1653000a-1363500a^2+369000a^3-398400s+663600as-272700a^2s \\
& -80340s^2+66600as^2-5400s^3
\end{aligned}$
\item $\begin{aligned}[t]
a_{12}=& 802635-2715000a+3386250a^2-1845000a^3+371115a^4+650302s-1641000as \\
& +1359000a^2s-369000a^3s+197555s^2-330600as^2+136350a^2s^2+26670s^3 \\
& -22200as^3+1350s^4.
\end{aligned}$
\end{enumerate}
\item[(iii)] For $f_{a,4,s,3,1}$ we have $M=1920$ and
\begin{enumerate}
\item $a_1=1683, a_2=6060, a_3=8770, a_4=16860, a_5=32400$
\item $a_6=-32580+24600a-6060 s, a_7=-90420+69000a-16860s$
\item $a_8=-173520+133200a-32400s$
\item $a_9=240490-371400a+140850a^2+90080s-69000as+8430s^2$
\item $a_{10}=460740-716400a+272700a^2+172860s-133200as+16200s^2$
\item $\begin{aligned}[t]
a_{11}=& -807900+1912200a-1473300a^2+369000a^3-457080s+714000as-272700a^2s \\
& -86100s^2+66600as^2-5400s^3
\end{aligned}$
\item $\begin{aligned}[t]
a_{12}=& 1051035-3375000a+3953850a^2-2001000a^3+371115a^4+797782s-1899000as \\
& +1468800a^2s-369000a^3s+226715s^2-355800as^2+136350a^2s^2+28590s^3 \\
& -22200as^3+1350s^4.
\end{aligned}$
\end{enumerate}
\end{itemize}
\end{lemma}
\begin{proof}
We sketch the proof, since it is similar to the one in \cite[Lemma A.7]{lr2}.

By Lemma \ref{tf0}(1) and (3) we see that there exists a symmetric polynomial $p_s \in \Q[x_1,\ldots,x_s]$ of degree at most $4$ such that
$$f=\frac{m_{1^s}(s)}{M}p_s.$$
Thus, we can express $p_s$ through the basis \eqref{base} as follows:
\begin{equation}
\label{bas}
\begin{split}
p_s=& a_1m_4(s)+a_2m_{31}(s)+a_3m_{22}(s)+a_4m_{211}(s)+a_5m_{1111}(s)+ a_6m_3(s)+\\
& +a_7m_{21}(s)+a_8m_{111}(s)+a_9m_2(s)+a_{10}m_{11}(s)+a_{11} m_1(s)+a_{12}
\end{split}
\end{equation}
with $a_1,\ldots,a_{12}\in \Q$. By Lemma \ref{tf0}(2), we have that
\begin{equation}
\label{g1-bis}
p_4(x_1,\ldots, x_4)= p_s(x_1,\ldots, x_4, 1 \ldots, 1).
\end{equation}
On the other hand, direct calculations show that:
\begin{equation}
\label{g2}
\begin{split}
f_{a,4,4,2,0} = \frac{m_{1^4}(4)}{360}[& 66m_4(4)+225m_{31}(4)+320m_{22}(4)+600m_{211}(4)+1125m_{1111}(4) \\
& +(825a-2025)m_3(4)+(2250a-5400)m_{21}(4)+(4275a-10125)m_{111}(4) \\
& +(4200a^2-20250a+24120)m_2(4)+(8025a^2-38475a+45225)m_{11}(4)\\
& +(9750a^3-72225a^2+172125a-133650)m_1(4)+8655a^4-87750a^3+323325a^2\\
& -510300a+293931].
\end{split}
\end{equation}
\begin{equation}
\label{g2-1}
\begin{split}
f_{a,4,4,3,0}=&\frac{m_{1^4}(4)}{1920}[1683m_{4}(4)+6060m_{31}(4)+8770m_{22}(4)+16860m_{211}(4)+32400m_{1111}(4)\\
& +(-54540+24600a)m_{3}(4)+(-151740+69000a)m_{21}(4)+(-291600+133200a)m_{111}(4)\\
& +(679770-621000a+140850a^2)m_{2}(4)+(1306260-1198800a+272700a^2)m_{11}(4)\\
& +(-3883140+5373000a-2454300a^2+369000a^3)m_{1}(4)\\
& +8617203-15989400a+11003850a^2-3321000a^3+371115a^4].
\end{split}
\end{equation}
\begin{equation}
\label{g2-2}
\begin{split}
f_{a,4,4,3,1}=&\frac{m_{1^4}(4)}{1920}[1683m_{4}(4)+6060m_{31}(4)+8770m_{22}(4)+16860m_{211}(4)+32400m_{1111}(4)\\
& +(-56820+24600a)m_{3}(4)+(-157860+69000a)m_{21}(4)+(-303120+133200a)m_{111}(4)\\
& +(735690-647400a+140850a^2)m_{2}(4)+(1411380-1249200a+272700a^2)m_{11}(4)\\
& +(-4359420+5833800a-2564100a^2+369000a^3)m_{1}(4)\\
& +10044963-18084600a+12010650a^2-3477000a^3+371115a^4].
\end{split}
\end{equation}
Hence, replacing $x_5=\ldots=x_s=1$ in \eqref{bas} and using \cite[Lemma A.6]{lr2} for $G=p_s$, we get an expression for $p_s(x_1,\ldots, x_4, 1 \ldots, 1)$ in terms of the basis \eqref{base} whose coefficients must coincide, by \eqref{g1-bis}, with the ones in \eqref{g2}-\eqref{g2-2}. Solving the corresponding linear systems in the $a_j$'s, we get (1)-(7) in (i)-(iii). 
\end{proof}

Consider now the following polynomials in $\Q[x_1,\ldots, x_s]$:

\begin{equation}
\label{gi}
\begin{split}
g_{a,4,s}& = \frac{m_{1^s}(s)}{1728}[ 25900-82800a+95380a^2-46800 a^3+8320a^4+21160s-50220as+38336a^2s\\
& -9360a^3s+6481s^2-10152as^2+3852a^2s^2+882s^3-684as^3+45s^4+(-21600+50760a\\
& -38520a^2+9360a^3-13140s+20412as-7704a^2s-2664s^2+2052as^2-180s^3)m_1(s)\\
& +(7100-10800a+4036a^2+2860s-2160as+288s^2)m_1(s)^2\\
& +(-1080+792a-216s)m_1(s)^3+64m_1(s)^4+(-880+1080a-368a^2-356s+216as\\
& -36s^2)m_{11}(s)+(360-216a+72s)m_1(s)m_{11}(s)-40m_1(s)^2m_{11}(s)+4m_{11}(s)^2]\\
\delta_s=&\frac{m_{1^s}(s)}{8}[145-300a+155a^2+59s-60as+6s^2+(-60+60a-12s)m_1(s)+7m_1(s)^2\\
& \hskip 1.2cm -2m_{11}(s)]\\
h_s=& -2f_{a,4,s,3,1}+2f_{a,4,s,3,0}+\delta_s\\
k_s=& 5(m_1(s)-s+3a-5)-5)h_s -\frac{25}{4}(m_1(s)-s+3a-5)^2\delta_s\\
c_s=& \frac{1}{8}[-1315+1800a-605a^2-523s+360as-52s^2+(520-360a+104s)m_1(s)\\
& \hskip .3cm -49m_1(s)^2-6m_{11}(s)]\delta_s +[4m_1(s)-4s-20+15a]h_s\\
\chi'_s=& \frac{k_s+c_s}{12}.
\end{split}
\end{equation}
Then we have
\begin{lemma}
\label{gl2} 
For all $s \ge 1$ the following identities hold:
\begin{enumerate}
\item $\begin{aligned}[t]
g_{a,4,s} = & \frac{5m_{1^s}(s)}{1728}[64m_4(s)+216m_{31}(s)+308m_{22}(s)+576m_{211}(s)+1080m_{1111}(s)+\\
& (-1080+792a-216s)m_3(s)+(-2880+2160a-576s)m_{21}(s)+\\
& (-5400+4104a-1080s)m_{111}(s)+\\
& +(7100-10800a+4036a^2+2860s-2160as+288s^2)m_2(s)+\\
& (13320-20520a+7704a^2+5364s-4104as+540s^2)m_{11}(s)+\\
& (-21600+50760a-38520a^2+9360a^3-13140s+20412as-7704a^2s-2664s^2+\\
& + 2052as^2-180s^3)m_1(s)+25900-82800a+95380a^2-46800a^3+8320a^4+21160s\\
& -50220as+38336a^2s-9360a^3s+6481s^2-10152as^2+3852a^2s^2+882s^3 
-684as^3+\\
& +45s^4]
\end{aligned}$
\item $\delta_s=\frac{m_{1^s}(s)}{8}[145-300a+155a^2+59s-60as+6s^2+(-60+60a-12s)m_{1}(s)+7m_{2}(s)+12m_{11}(s)]$
\item $\begin{aligned}[t]
h_s =& \frac{m_{1^s}(s)}{8}[19m_3(s)+51m_{21}(s)+96m_{111}(s)+(-255+220a-51s)m_2(s)\\
& +(-480+420a-96s)m_{11}(s)+(1185-2100a+915a^2+477s-420as+48s^2)m_1(s)\\
& -1925+5200a-4575a^2+1300a^3-1170s+2090as-915a^2s-237s^2+210as^2-16s^3]
\end{aligned}$
\item $\begin{aligned}[t]
k_s=& \frac{5m_{1^s}(s)}{32}[41m_4(s)+150m_{31}(s)+218m_{22}(s)+422m_{211}(s)+816m_{1111}(s)\\
& +(-750+598a-150s)m_3(s)+(-2110+1702a-422s)m_{21}(s)\\
& +(-4080+3312a-816s)m_{111}(s)\\ 
&+(5240-8510a+3410a^2+2103s-1702as+211s^2)m_2(s)\\
& +(10130-16560a+6670a^2+4066s-3312as+408s^2)m_{11}(s)\\
& +(-16650+41170a-33350a^2+8830a^3-10060s+16514as-6670a^2s-2026s^2+1656as^2\\
& -136s^3)m_1(s)+20375-67850a+83000a^2-44150a^3+8625a^4+16475s-40940as\\
& +33275a^2s-8830a^3s+4995s^2-8234as^2+3335a^2s^2+673s^3-552as^3+34s^4]
\end{aligned}$
\item $\begin{aligned}[t]
c_s=& \frac{m_{1^s}(s)}{64}[265m_4(s)+924m_{31}(s)+1330m_{22}(s)+2524m_{211}(s)+4800m_{1111}(s)\\
& +(-4620+3860a-924s)m_3(s)+(-12620+10580a-2524s)m_{21}(s)\\
& +(-24000+20160a-4800s)m_{111}(s)\\
& +(31210-52900a+22250a^2+12552s-10580as+1262s^2)m_2(s)\\
& +(59380-100800a+42380a^2+23876s-20160as+2400s^2)m_{11}(s)\\
& +(-96900+249500a-211900a^2+59300a^3-58760s+100300as-42380a^2s-11876s^2\\
& +10080as^2-800s^3)m_1(s)+117325-407500a+524450a^2-296500a^3+62225a^4+95380s\\
& -247000as+210840a^2s-59300a^3s+29073s^2-49900as^2+21190a^2s^2+3938s^3-3360as^3\\
& +200s^4]
\end{aligned}$
\item $\begin{aligned}[t]
\chi'_s=& \frac{m_{1^s}(s)}{768}[675m_4(s)+2424m_{31}(s)+3510m_{22}(s)+6744m_{211}(s)+12960m_{1111}(s)\\
& +(-12120+9840a-2424s)m_3(s)+(-33720+27600a-6744s)m_{21}(s)\\
& +(-64800+53280a-12960s)m_{111}(s)\\
& +(83610-138000a+56350a^2+33582s-27600as+3372s^2)m_2(s)\\
& +(160680-266400a+109080a^2+64536s-53280as+6480s^2)m_{11}(s)\\
& +(-263400+661200a-545400a^2+147600a^3-159360s+265440as-109080a^2s\\
& -32136s^2+26640as^2-2160s^3)m_1(s)+321075-1086000a+1354450a^2-738000a^3\\
& +148475a^4+260130s-656400as+543590a^2s-147600a^3s+79023s^2-132240as^2\\
& +54540a^2s^2+10668s^3-8880as^3+540s^4]. 
\end{aligned}$            
\end{enumerate}
\end{lemma}
\begin{proof}
Immediate from \cite[Lemma A.5]{lr2}.
\end{proof}

We wish to compare all of the above functions. 

In order to do this, for any $a, b \in \Z$, define the polynomial $v_{s,a,b} \in \Q[x_1,\ldots,x_s]$ by 
$$v_{s,a,b} =bm_4(s)+10m_{22}(s)+(50a^2-10s-50)m_2(s)-250a^2-50a^2s+5s^2+150+(55-b)s-5b+(100+5b)a^4.$$
We have
\begin{lemma}
\label{gl4} 

For all $s \ge 4$ the following identities hold:
\begin{itemize}
\item[(i)] $g_{a,4,s}-f_{a,4,s,2,0}=\frac{m_{1^s}(s)}{4320}v_{s,a,8}$.
\item[(ii)] $\chi'_s-f_{a,4,s,3,0}=\frac{m_{1^s}(s)}{3840}v_{s,a,9}$.
\end{itemize}
\end{lemma}
\begin{proof}
Immediate from Lemmas \ref{gl1} and \ref{gl2}.
\end{proof}

We will also need the following crude estimate:

\begin{lemma}
\label{cg}
Let $b \in \{8,9\}$, let $a \ge 2$ and let $s \ge 2$ be integers. Then for all integers $d_i \ge 1, 1 \le i \le s$ we have that $v_{s,a,b}(d_1,\ldots, d_s) > 0$.
\end{lemma}
\begin{proof}
Note that 
$$v_{s,1,b}= bm_4(s)+10m_{22}(s)-10sm_2(s)+s(5s-b+5).$$
In the notation in \cite[Lemma A.10]{lr2} we have that $v_{s,1,b}=q_{s,b}$ and it follows by \cite[Lemma A.10]{lr2} that $v_{s,1,b}(d_1,\ldots,d_s)> 0$ if $\prod_{i=1}^s d_i \ge 2$. On the other hand, if $\prod_{i=1}^s d_i =1$, then one easily checks that $v_{s,1,b}(1,\ldots,1)=0$. Thus 
\begin{equation}
\label{ult1}
v_{s,1,b}(d_1,\ldots,d_s) \ge 0.
\end{equation}
Now observe that
\begin{equation}
\label{ult}
v_{s,a,b}(d_1,\ldots,d_s)=v_{s,a-1,b}(d_1,\ldots,d_s)+5(2a-1)[10(m_2(s)(d_1,\ldots,d_s)-s)+b(2a^2-2a+1)+10(4a^2- 4a-3)].
\end{equation}
Finally, we proceed by induction on $a$. If $a=2$, \eqref{ult} and \eqref{ult1} give
$$v_{s,2,b}(d_1,\ldots,d_s)=v_{s,1,b}(d_1,\ldots,d_s)+75[2(m_2(s)(d_1,\ldots,d_s)-s)+b+10]>0.$$
If $a \ge 3$ we have by induction that $v_{s,a-1,b}(d_1,\ldots,d_s)>0$ and therefore also $v_{s,a,b}(d_1,\ldots, d_s) > 0$ by \eqref{ult}.
\end{proof}

\end{document}